\documentclass[
  a4paper, 
  reqno, 
  oneside, 
  11pt
]{amsart}

\usepackage[dvipsnames]{xcolor} 
\colorlet{cite}{LimeGreen!50!Green}
\usepackage{tikz}  
\usepackage{ circuitikz }
\usetikzlibrary{matrix,arrows,arrows.meta,
decorations.markings,decorations.pathmorphing,calc,intersections,shapes,backgrounds,positioning}
\tikzset{ 
  baseline=-2.3pt,
  text height=1.5ex, text depth=0.25ex,
  >=stealth,
  node distance=2cm,
  mid/.style={fill=white,inner sep=2.5pt},
}
\usepackage{float}
\usepackage{lmodern}
\usepackage{tikz-cd}
\usepackage{amsthm, amssymb, amsfonts}
\usepackage[all]{xy}
\newcommand{\mn}{\mathfrak n }
\newcommand{\ma}{\mathfrak a }
\usepackage{arydshln}
\usepackage{graphicx,caption,subcaption}
\usepackage{braket}  
\usepackage{microtype}
\usepackage[section]{placeins}

\usepackage[%
  bookmarks=true,			
  unicode=true,			
  pdftitle={Deformations of noncompact manifolds},		%
  pdfauthor={Elizabeth Gasparim}{Carlos Varea},	%
  pdfkeywords={}{adjoint orbit}{correspondence}{},	
  colorlinks=true,		
  linkcolor=Blue,			
  citecolor=cite,		
  filecolor=magenta,		
  urlcolor=RoyalBlue			
]{hyperref}				
\usepackage{cleveref}

\newtheoremstyle{mydef}
  {}		
  {}		
  {}		
  {}		
  {\scshape}	
  {. }		
  { }		
  {\thmname{#1}\thmnumber{ #2}\thmnote{ #3}}	

\theoremstyle{plain}	
\newtheorem{theorem}{Theorem} 
\newtheorem{proposition}[theorem]{Proposition}
\newtheorem*{theorem*}{Theorem}
\newtheorem*{proposition*}{Proposition}
\newtheorem{corollary}[theorem]{Corollary}

\theoremstyle{mydef} 
\newtheorem{definition}[theorem]{Definition}
\newtheorem{example}[theorem]{Example}
\newtheorem*{conjecture*}{Conjecture}
\theoremstyle{remark}
\newtheorem{remark}[theorem]{Remark}

\DeclareMathOperator{\tr}{tr}

\DeclareMathOperator{\Ad}{Ad}
\DeclareMathOperator{\ad}{ad}
\DeclareMathOperator{\htt}{ht}
\DeclareMathOperator{\Lie}{Lie}

\newtheorem*{lemma*}{Lemma}
\newtheorem*{corollary*}{Corollary}
\theoremstyle{definition}
\theoremstyle{remark}

\author{Elizabeth  Gasparim, Lino Grama {\tiny and} Carlos Varea}
\address{E.G. - Depto. Matem\'aticas, Univ. Cat\'olica del Norte, Antofagasta, Chile. 
L.G. - Instituto de Matemática, Unicamp, Brazil. 
C.V. -  Depto. Matem\'atica, Univ. Tecnológica Federal do Paraná, Brazil.
etgasparim@gmail.com, linograma@gmail.com  carlosbassanivarea@gmail.com}
\title[Infinitesimal $T$-duality as a source of $H$-flux]{The $\mathbf H$-flux  on flag manifolds generated by infinitesimal $\mathbf T$-duality}

\begin{document}

\begin{abstract}
We define a new  correspondence  for pairs $(\mathbb{F},H)$ formed by a flag manifold $\mathbb{F}$ together with an 
 $H$-flux on $\mathbb{F}$.
  Given its role within our correspondence,  infinitesimal $T$-duality may be viewed as a source of $H$-flux,
in the sense that it contributes towards taking  fluxless pairs $(\mathbb{F},0)$  to 
pairs $(\mathbb{F}^\vee, H^\vee)$ carrying nontrivial flux $H^\vee\neq 0$.  We also illustrate how our correspondence exchanges
complex structures with symplectic ones up to $B$-transformations.
\end{abstract}
\maketitle
\tableofcontents

\section{Motivation: generalized  complex geometry and $T$-duality}

The interaction between complex geometry and symplectic geometry motivated by the mirror symmetry program has been a significant source of challenging problems in differential geometry. The language of generalized complex structures first introduced by Hitchin \cite{H} has the property of encompassing both complex and symplectic geometry as particular cases. 
Therefore, the use of generalized complex geometry in mirror symmetry problems becomes quite natural. In the context
of generalized complex structures,  Cavalcanti and Gualtieri  \cite{CG1,CG2, G1,G2} employed the notion of $T$-duality 
(in the sense considered by Bouwknegt, Evslin, Hannabuss, Mathai, Rosenberg, Wu \cite{BEM1,BEM2, BHM1,BHM2, MR, MW}) to establish an isomorphism between Courant algebroids of two $T$-dual manifolds. Consequently, each generalized complex structure on a differential manifold $M$ has a corresponding 
generalized complex structure on the $T$-dual manifold $M^\vee$. In some specific cases, Cavalcanti and Gualtieri 
 demonstrate the occurrence of complex--symplectic duality which can be understood as a manifestation of 
 Strominger--Yau--Zaslow mirror symmetry \cite{SYZ}.

Del Barco, Grama, and Soriani \cite{BGS} extended the work of  Cavalcanti and Gualtieri \cite{CG1} to the  case of nilmanifolds. 
The fundamental  step in their construction was the introduction of infinitesimal $T$-duality, serving as an analog to $T$-duality at the level of the Lie algebra of a nilpotent group.

In this work, we propose a correspondence between flag manifolds endowed with 3-form fluxes (to which we refer as {\it flowing flags}), drawing inspiration from ideas presented in the 
literature of generalized complex structure and of $T$-duality such as the works  by Cavalcanti--Gualtieri,
 Bouwknegt--Hannabuss--Mathai, and del Barco--Grama--Soriani. We utilize the infinitesimal $T$-duality of nilpotent algebras to induce a correspondence between generalized flag manifolds with fluxes. This correspondence emerges from the nilpotent radicals of the parabolic subalgebras used in constructing flag manifolds. The resulting process gives rise to a new parabolic subalgebra and, consequently, a new  flag manifold together with  a flux. 

We thoroughly explore  geometric and algebraic (Lie-theoretical) aspects of this construction while also presenting various concrete situations where a phenomenon pertaining to mirror symmetry occurs, in the sense of an interchange of complex and symplectic structures between the corresponding flowing flag manifolds.

In the following section we  provide a detailed walkthrough of how our correspondence is constructed within the context of flag manifolds. To illustrate our approach, we first focus on the particular case of the flag manifold $\mathbb F(1,2)$, which represents complete flags in the three-dimensional complex space $\mathbb{C}^3$. We will then establish a correspondence  between this flag manifold and the complex projective space $\mathbb{CP}^3$ with a flux.
This basic example offers a concrete view of the practical implementation of our correspondence within the context of 
flowing flag manifolds, illustrating  the groundwork for the subsequent theoretical development and the proof of the theorems we will 
implement in the remainder of the paper.\\

\paragraph {\bf Acknowledgements:} We are grateful for the generosity of   Luiz A. B. San Martin 
for the continued flow of good ideas offered to us,
in particular for the ones that generated this work.  We thank Christian Saemann for useful suggestions. 
We thank the Abdus Salam International Centre for Theoretical Physics for supporting our meetings on Deformation 
Theory; our cooperation  started out on the first meeting
and this text was prepared for the upcoming second meeting. L. Grama was partially supported by S\~ao Paulo Research Foundation FAPESP grants 2018/13481-0, 2021/04003-0, 2021/04065-6.

\section{Our proposal: a new correspondence for flowing flags}
We use infinitesimal $T$-duality to define a new correspondence between flag manifolds in the following sense:
\begin{enumerate}
\item Starting with a pair $(\mathbb{F}={G}/{P},H) $ formed by a flag manifold and an $H$-flux, 
\item  our new correspondence proceeds via  infinitesimal $T$-duality to  take
   the nilradical $\mathfrak n$ of $\mathfrak p = \Lie(P)$ 
 to its dual $\mathfrak n^\vee$, next
 \item we choose a group $\mathfrak g^\vee$ of the same type as $\mathfrak g$, then
 \item we choose a parabolic subalgebra $\mathfrak p^\vee$ of $\mathfrak g^\vee$, then
 \item we exchange the original fluxes and  nilradicals,   generating
 a new flag manifold $\mathbb{F}^\vee={G^\vee}/{P^\vee}$ endowed with a flux $H^\vee$, and
 \item provide several examples where $H$ vanishes whereas $H^\vee$ is nontrivial, therefore  generating flux on $\mathbb{F}^\vee$ out of the fluxless flag $\mathbb{F}$.
\end{enumerate}

\begin{center}
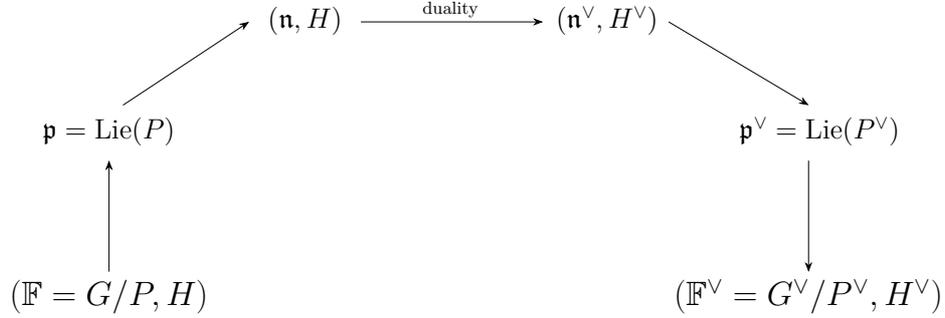
\begin{figure}[H]
\resizebox{1\textwidth}{!}{
\begin{tikzpicture}
\node [font=\LARGE] at (3.5,0.75) {$(\mathbb{F}=G/P, H)$};
\draw [ -Stealth] (3.5,1.25) -- (3.5,3.25);
\node [font=\Large] at (3.5,3.75) {$\mathfrak{p}=\Lie(P)$};
\draw [ -Stealth] (3.75,4.25) -- (6,5.75);
\node [font=\Large] at (7,5.75) {$(\mathfrak{n},H)$};
\draw [-Stealth] (8,5.75) -- (11.25,5.75);
\node [font=\Large] at (12.4,5.75) {$(\mathfrak{n}^\vee,H^\vee)$};
\draw [ -Stealth] (13.5,5.75) -- (16,4.25);
\node [font=\Large] at (16.2,3.75) {$\mathfrak{p}^\vee=\Lie(P^\vee)$};
\draw [ -Stealth] (16,3.25) -- (16,1.25);
\node [font=\LARGE] at (16,0.75) {$(\mathbb{F}^\vee=G^\vee/P^\vee, H^\vee)$};
\node [font=\footnotesize] at (9.6,6) {duality};
\end{tikzpicture}
}
\caption{The flowing flags correspondence}
\label{fig-def}
\end{figure}
\end{center}

Some remarks are worth mentioning:
\begin{itemize}
\item[$\iota.$] The dimensions of $\mathfrak{n}$ and $\mathfrak{n}^\vee$ are the same;
\item[$\iota \iota$.] The step from $\mathfrak{n}^\vee$ to $\mathfrak{p}^\vee$ is not unique. More than one choice of $\mathfrak{p}^\vee$ may complete the diagram.
\item[$\iota \iota \iota$.] $G^\vee$ has the same type as $G$ in the Cartan classification of simple Lie groups (i.e., $A_\ell, B_\ell, C_\ell$, and so on), and
\item[$\iota \nu$.] $\dim G^\vee/P^\vee = \dim \mathfrak{n}^\vee$.
\end{itemize}

\begin{example}\label{e1}
Consider the maximal flag manifold  of the group $G^\mathbb{C} = SL(3,\mathbb{C})$, that is, the homogeneous manifold $\mathbb F=SL(3,\mathbb{C})/P$, where $P$ is a minimal parabolic subgroup. Note that $\mathfrak{g}^\mathbb{C} = \mathfrak{sl}(3,\mathbb{C})$, so we have the following data:
\begin{itemize}
\item[$\iota$.] the system of simple roots: $\Sigma = \{ \alpha, \beta\}$,
\item[$\iota \iota$.] the set of positive roots: $\Pi^+ = \{ \alpha, \beta, \alpha+\beta\}$,  
\item[$\iota \iota \iota$.] the root spaces: $\mathfrak{g}_\alpha = \langle E_{12}\rangle$, $\mathfrak{g}_\beta = \langle E_{23}\rangle$, $\mathfrak{g}_{\alpha+\beta} = \langle E_{13}\rangle$, and 
\item[$\iota \nu$.] the Cartan subalgebra: $\mathfrak{h} = \langle E_{11} - E_{22}, E_{22}-E_{33}\rangle$.
\end{itemize}
Using  Malcev's notation, we write  $\mathfrak{n} = (0,0,-e^{12})$ and $\mathfrak{a} = ( e_3)$. We now consider the admissible triple  $(\mathfrak{n},\mathfrak{a}, H=0)$.
 Using Theorem \ref{nildual}, we see that the dual triple is $(\mathfrak{n}^\vee,\mathfrak{a}^\vee, H^\vee)$ where $\mathfrak{n}^\vee = (0,0,0)$ (that is, $\mathfrak{n}^\vee$ is abelian), $\mathfrak{a}^\vee = \mathbb{R}$ and $H^\vee = (-e^{123})$. 

\begin{center}
\begin{figure}[H]
\begin{tikzpicture}
\node at (1,0) {$H=0$};
\node at (8.5,0) {$H^\vee=-e^{123}$};
\node at (1,3) {$\mn=(0,0,-e^{12}),\, \ma =(e_3) $};
\node at (8.5,3) {$\mn^\vee=(0,0,0),\, \ma =(\tilde e_3) $};
\draw [->] (1.1,2.85) to [out=-57, in=135](8.9,0.35);
\draw [dashed, ->] (1.45,0.30) to [out=40,in=-130](8.5,2.85);

\end{tikzpicture}
\caption{Infinitesimal duality}
\end{figure}
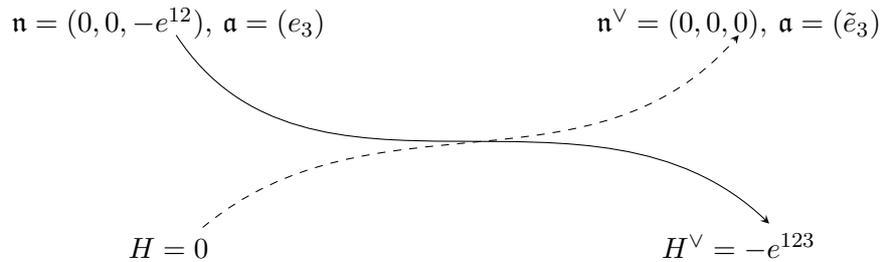
\end{center}

Since  $\mathfrak{n}^\vee$ is the nilradical of the parabolic subalgebra $\mathfrak{p}^\vee = \langle E_{14},E_{24},E_{34}\rangle$, we have that  $P_\Theta$ is the parabolic subgroup corresponding to the roots $\Theta = \{ \alpha_1,\alpha_2 \}$, and we obtain the flag manifold 
 $$\mathbb{F}^\vee = G^\vee/P^\vee = SL(4,\mathbb{C})/P_\Theta \simeq \mathbb{CP}^3.$$
Therefore, we conclude that $\mathbb{F}(1,2)$ corresponds to $(\mathbb{CP}^3,H^\vee)$.
\end{example}

\section{Extending infinitesimal T-duality to  flag manifolds }

\noindent {\bf T-duality.} Cavalcanti and Gualtieri \cite{CG2} established a link between $T$-duality and generalized complex structures. They realized $T$-duality as an isomorphism between the underlying orthogonal and Courant structures of two possibly topologically distinct manifolds. In certain cases such a duality can interchange complex and symplectic structures between the $T$-dual manifolds. 

The definition of topological $T$-duality used in \cite{CG2} for  torus bundles equipped with a closed $3$-form is the following:

\begin{definition}
Let $E$ and $E^\vee$ be principal fibre bundles over the same base $M$ with $k$-dimensional tori $T$ and $T^\vee$ as the fibres,
together  with closed invariant $3$-forms $H\in \Omega ^3 (E)$, $H^\vee \in \Omega^3 (E^\vee)$. Let $E\times _M E^\vee$ be the product bundle and consider the diagram
\[
\xymatrix{
& \ar[dl]_p E\times E^{\vee} \ar[dr]^{p^\vee}\\
E \ar[dr]& & \ar[dl] E^\vee\\
& M &
}
\]

We say that  $(E,H)$ and $(E^\vee ,H^\vee )$ are {\bf $T$-dual} if $p^\ast H-p^{\vee \ast}H^\vee = dF$,
for some $2$-form  $F\in \Omega ^2 (E\times_M E^\vee)$ that is $T\times_M T^\vee$-invariant and non-degenerate on the fibres. The product $E\times_M E^\vee$ is called the
{\bf correspondence space}.
\end{definition}

Given a pair of  $T$-dual torus bundles $(E,H)$ and $(E^\vee, H^\vee)$, Cavalcanti and Gualtieri  \cite{CG2} defined an isomorphism between the space of invariant sections 
\[
\varphi \colon (TE\oplus T^\ast E)/T \rightarrow (TE^\vee \oplus T^\ast E^\vee)/T^\vee.
\] 
Using the $2$-form $F$ and the correspondence space of $T$-dual pairs, this isomorphism $\varphi$ is given explicitly by
\[
\varphi (X+ \xi) = p_{\ast} ^\vee (\widehat{X}) + p^\ast \xi - F(\widehat{X}),
\]
where $\widehat{X}$ is the unique lift of $X$ to $E\times _M E^\vee$ such that $p^\ast \xi - F(\widehat{X})$ is a basic $1$-form.
Such an isomorphism preserves the natural bilinear forms, the Courant brackets twisted by the $3$-forms $H$ and $H^\vee$,
and   exchanges $T$-invariant generalized complex structures between $E$ and $E^\vee$.\\

$T$-duality has been studied in several contexts, in particular, del Barco, Grama, and Soriani \cite{BGS} studied $T$-duality in the context of nilmanifolds, that is,   homogeneous compact manifolds which are quotients of  nilpotent Lie groups
by a discrete cocompact subgroup (a lattice).  Using the structure of torus bundle naturally 
 occurring on nilmanifolds,  \cite{BGS} explored an infinitesimal version of the $T$-duality on them, 
 which we now summarize.

Consider a pair of Lie algebras $\mathfrak{g}$ and $\mathfrak{g}^\vee$ that are isomorphic, up to  quotienting by abelian ideals, that is, there are abelian ideals $\mathfrak{a}$ and $\mathfrak{a}^\vee$ such that $\mathfrak{g}/\mathfrak{a} \simeq \mathfrak{g}^\vee /\mathfrak{a}^\vee$. Let us denote by $\mathfrak{n}$ the quotient Lie algebra and by $q\colon \mathfrak{g}\rightarrow \mathfrak{n}$ and $q^\vee \colon \mathfrak{g}^\vee \rightarrow \mathfrak{n}$ the quotient maps.

The subset $\mathfrak{c}$ of $\mathfrak{g}\oplus \mathfrak{g}^\vee$, defined by
\[
\mathfrak{c} = \{ (x,y)\in \mathfrak{g}\oplus \mathfrak{g}^\vee \ | \ q(x) = q^\vee (y)\}
\]
is a Lie algebra and the following diagram is commutative
\[
\xymatrix{
& \ar[dl]_p \mathfrak{c} \ar[dr]^{p^\vee}\\
\mathfrak{g} \ar[dr]_q & & \ar[dl]^{q^\vee} \mathfrak{g}^\vee\\
& \mathfrak{n} &
}
\]
where $p$ and $p^\vee$ are the projections over the first and second component, respectively. The Lie subalgebras $\mathfrak{k} = \{ (x,0)\in \mathfrak{c} \ | \ x\in \mathfrak{a}\}$ and $\mathfrak{k}^\vee = \{ (0,y)\in \mathfrak{c} \ | \ y\in \mathfrak{a}^\vee\}$ are abelian ideals of $\mathfrak{c}$. In particular, $\mathfrak{c}/\mathfrak{k}^\vee \simeq \mathfrak{g}$ and $\mathfrak{c}/\mathfrak{k} \simeq \mathfrak{g}^\vee$.

\begin{definition}Let $\mathfrak{g}$ be a Lie algebra together with a closed $3$-form $H$. Let $\mathfrak{a}$  be an abelian ideal of $\mathfrak{g}$, we say that the triple $(\mathfrak{g}, \mathfrak{a}, H)$ is {\bf admissible} if $H(x,y,\cdot ) = 0$ for all $x,y \in \mathfrak{a}$.
\end{definition}

\begin{definition}
For $\mathfrak{a} =\Lie(T)$, two admissible triples $(\mathfrak{g}, \mathfrak{a}, H)$ and $(\mathfrak{g}^\vee, \mathfrak{a}^\vee, H^\vee)$ are said to be {\bf infinitesimally $T$-dual} if $\mathfrak{g}/\mathfrak{a} \simeq \mathfrak{g}^\vee/ \mathfrak{a}^\vee$ and there exists a $2$-form $F$ in $\mathfrak{c}$ which is non-degenerate in the fibres such that $p^\ast H - p^{\vee \ast} H^\vee = dF$. 
\end{definition}

Using this definition,  the following theorem about the existence and uniqueness of dual triples  proved in \cite{BGS}, shows 
in addition, 
 how to construct the dual triple.

\begin{theorem}\label{nildual}\cite[Thm.\thinspace 3.4]{BGS}
Let $(\mathfrak{g}, \mathfrak{a}, H)$ be an admissible triple with $\mathfrak{a}$ a central ideal and let $\{ x_1,\cdots, x_m\}$ be a basis of $\mathfrak{a}$. Let $\mathfrak{a}^\vee = \mathbb{R}^m$ and define
\begin{itemize}
\item $\Psi ^\vee \colon \mathfrak{n}\times \mathfrak{n} \rightarrow \mathfrak{a}^\vee$ given by $\Psi ^\vee = (\iota_{x_1}H, \iota_{x_2}H,\cdots , \iota_{x_m}H)$,

\item $\mathfrak{g}^\vee = (\mathfrak{g}/\mathfrak{a})_{\Psi ^\vee}$ and

\item $H^\vee = \sum_{k=1} ^m z^k \wedge dx^k + \delta$ where $\{ z_1,\cdots , z_m\}$ is a basis of $\mathfrak{a}^\vee$ and $\delta$ is the basic component of $H$.
\end{itemize}
Then $(\mathfrak{g}^\vee, \mathfrak{a}^\vee, H^\vee)$ is an admissible triple and is dual to $(\mathfrak{g}, \mathfrak{a}, H)$.

Conversely, if $(\mathfrak{g}^\vee, \mathfrak{a}^\vee, H^\vee)$ is dual to $(\mathfrak{g}, \mathfrak{a}, H)$, then there exist a basis $\{ x_1,\cdots , x_m\}$ of $\mathfrak{a}$ and a basis $\{z_1,\cdots ,z_m\}$ of $\mathfrak{a}^\vee$ such that the formulas above hold.
\end{theorem}

\noindent {\bf Flag manifolds.} 
Let $\mathfrak{g}^\mathbb{C}$ be a complex semi-simple Lie algebra and consider 
 a Cartan subalgebra $\mathfrak{h}$ of $\mathfrak{g}^\mathbb{C}$.   The set $\Pi$ 
 of all roots associated to the pair $(\mathfrak{g}^\mathbb{C}, \mathfrak{h})$, decomposes as  
\[
\mathfrak{g}^\mathbb{C} = \mathfrak{h} \oplus \sum_{\alpha \in \Pi} \mathfrak{g}_\alpha
\]
where $\mathfrak{g}_\alpha = \{ X \in \mathfrak{g}^\mathbb{C} : \forall H \in \mathfrak{h}, [H,X] = \alpha (H) X\}$ is the root space associated to the root $\alpha$.

The {\bf Cartan--Killing form} is defined by
\[
\langle X,Y \rangle = \tr (\ad (X) \ad (Y))
\]
and its restriction to $\mathfrak{h}$ is non-degenerate. Given a root $\alpha \in \mathfrak{h}^\ast$ we can define $H_\alpha$ by $\alpha (\cdot ) = \langle H_\alpha , \cdot \rangle$ and denote $\mathfrak{h}_\mathbb{R} = \textnormal{span}_\mathbb{R} \{H_\alpha : \alpha \in \Pi\}$.

Let us fix a Weyl basis of $\mathfrak{g}^\mathbb{C}$ which amounts to giving $X_\alpha \in \mathfrak{g}_\alpha$ such that $\langle X_\alpha , X_{-\alpha}\rangle = 1$ and 
\[
[X_\alpha , X_\beta] = \left\lbrace  \begin{array}{ll}
m_{\alpha,\beta} X_{\alpha+\beta},& \ \textnormal{if } \alpha+\beta \in \Pi \\
0,& \ \textnormal{if } \alpha+\beta \notin \Pi
\end{array}\right.
\]
with $m_{\alpha,\beta} \in \mathbb{R}$ satisfying $m_{\alpha,\beta} = -m_{-\alpha,-\beta}$.

Let $\Pi^+ \subset \Pi$ be a choice of positive roots and $\Sigma$ be the corresponding simple root system. Given a subset $\Theta \subset \Sigma$, we will denote by $\langle \Theta \rangle$ the set of roots spanned by $\Theta$, $\Pi_M = \Pi \backslash \langle \Theta \rangle$ the set of complementary roots and $\Pi_M ^+$ the set of complementary positive roots.

Let
\[
\mathfrak{p}_\Theta = \mathfrak{h} \oplus \sum_{\alpha \in \langle \Theta \rangle} \mathfrak{g}_\alpha \oplus \sum_{\beta \in \Pi_M ^+} \mathfrak{g}_\beta
\]
be the parabolic subalgebra of $\mathfrak{g}^\mathbb{C}$ determined by $\Theta$.

\begin{definition} The {\bf flag manifold} $\mathbb{F}_\Theta$ is defined as the homogeneous space
\[
\mathbb{F}_\Theta = G^\mathbb{C}/P_{\Theta}
\]
where $G^\mathbb{C}$ is a complex connected Lie group with Lie algebra $\mathfrak{g}^\mathbb{C}$ and $P_\Theta$ is the normalizer of $\mathfrak{p}_\Theta$ in $G^\mathbb{C}$.
\end{definition}

We will use the simplified notation $\mathbb{F}$ omitting the subscript.

Let $\mathfrak{g}$ be a compact real form of $\mathfrak{g}^\mathbb{C}$. We have 
\[
\mathfrak{g} = \textnormal{span}_\mathbb{R} \{ i\mathfrak{h}_\mathbb{R}, A_\alpha, iS_\alpha : \alpha \in \Pi\},
\]
where $A_\alpha = X_\alpha - X_{-\alpha}$ and $S_\alpha = X_\alpha + X_{-\alpha}$.

Denote by $G$ the compact real form of $G^\mathbb{C}$ with $\Lie(G) = \mathfrak{g}$ and let $\mathfrak{k}_\Theta$ be the Lie algebra of $K_\Theta := P_\Theta \cap G$. We have
\[
\mathfrak{k}_\Theta ^\mathbb{C} = \mathfrak{h} \oplus \sum _{\alpha\in \langle \Theta \rangle} \mathfrak{g}_\alpha
\]
where $\mathfrak{k}_\Theta ^\mathbb{C}$ denotes the complexification of the real Lie algebra $\mathfrak{k}_\Theta = \mathfrak{g}\cap \mathfrak{p}_\Theta$.

The Lie group $G$ acts transitively on $\mathbb{F}$ and we have
\[
\mathbb{F} = G^\mathbb{C}/P_{\Theta} = G/(P_\Theta \cap G) = G/K_\Theta.
\]
When $\Theta = \emptyset$ the flag manifold $\mathbb{F} = G^\mathbb{C}/P = G/T$ is called {\bf maximal}, where $T = G\cap P$ is a maximal torus of $G$.\\

\noindent {\bf Isotropy representation.}
An important tool for the study of invariant tensors on the flag manifold $G/K_\Theta$ is the {isotropy representation}. 
Let us recall its definition. For $a\in G$, let $\tau_a:G/K_\Theta \to G/K_\Theta$
 be the left translation induced by the $G$-action on $G/K_\Theta$, that is, $\tau_a(gK_\Theta)=agK_\Theta$. We define the origin $o$ of $G/K_\Theta$ as the trivial coset, $o=eK_\Theta$. Note that for every $k\in K_\Theta$ we have $\tau_k(o)=o$.  
\begin{definition}
The {\bf isotropy representation} of the flag manifold $G/K_\Theta$ is the homomorphism 
\begin{equation}
\rho:K_\Theta\to GL(T_o (G/K_\Theta))
\end{equation}
defined by $\rho(k)(X)=(d\tau_k)_o(X)$, $k\in K_\Theta$ and $X\in T_o (G/K_\Theta) $.
\end{definition}


\begin{definition} The homogeneous space $G/K_\Theta$ is called {\bf reductive} if the Lie algebra of $G$ decomposes into $\mathfrak{g}= \mathfrak{k}_\Theta \oplus \mathfrak{m}$, where $\mathfrak{k}_\Theta=\Lie(K_\Theta)$, with $\Ad(k)\mathfrak{m}\subset \mathfrak{m}$, for all $k\in K_\Theta$. 

In such a case one can identify $T_o (G/K_\Theta)$ with $\mathfrak{m}$. 
\end{definition}

\begin{proposition}\cite{arv01}
The isotropy representation of the reductive homogeneous space $G/K_\Theta$ is equivalent (as representation) to $$\Ad\mid_{K_\Theta}:\mathfrak{m}\longrightarrow\mathfrak{m}.$$
\end{proposition}


In the case of flag manifolds the isotropy representation decomposes $\mathfrak{m}$ into irreducible components (see \cite{arv01}), that is, 
$$
\mathfrak{m}=\mathfrak{m}_{1}\oplus\mathfrak{m}_{2}\oplus\cdots\oplus\mathfrak{m}_{n},
$$
where each component $\mathfrak{m}_{i}$ satisfies $\Ad(K_\Theta)(\mathfrak{m}_{i})\subset \mathfrak{m}_{i}$. 
Moreover, each component $\mathfrak{m}_{i}$ is irreducible, that is, the only invariant subspaces of $\mathfrak{m}_{i}$ by $\Ad(K_\Theta)|_{\mathfrak{m}_{i}}$ are the trivial subspaces.

\begin{definition}
 We call the subspaces $\mathfrak{m}_{i}$ of $\mathfrak{m}$  the {\bf isotropic summands} of the isotropy representation
 of $G/K_\Theta$.
\end{definition}

\section{Flux generated by the flowing flags correspondence}

\begin{definition}
A {\bf flowing flag} is a pair $(\mathbb{F},H)$ consisting of a flag manifold $\mathbb{F} = G/P$ together with an invariant $3$-form $H$.
\end{definition}

We now define a correspondence in the set of flowing flags.
To each flag manifold we can associate a unique nilpotent Lie algebra. 
Indeed, since the flag manifolds $\mathbb{F} = G/P$ are in one-to-one correspondence with the parabolic subalgebras $\mathfrak{p}$ of $\mathfrak{g}$, 
it suffices to consider  the nilradical $\mathfrak{n}$ of $\mathfrak{p}$.

\begin{definition}\label{dualflag}
Consider the flowing flag $(\mathbb{F}=G/P,H)$ and let $\mathfrak{n}$ be the nilradical of the parabolic subalgebra $\mathfrak{p}=\Lie(P)$. We say that $(\mathbb{F},H)$ {\bf corresponds} to $(\mathbb{F}^\vee,H^\vee)$ if:  
\begin{itemize}
\item $(\mathfrak{n}^\vee,\mathfrak{a}^\vee, H^\vee)$ is the infinitesimal $T$-dual of $(\mathfrak{n}, \mathfrak{a},H)$, so that $\mathfrak{n}/\mathfrak{a} \simeq  \mathfrak{n}^\vee / \mathfrak{a}^\vee$ (where $\mathfrak{a}$ and $\mathfrak{a}^\vee$ are abelian),
\item $P^\vee$ is the parabolic subgroup associated to the parabolic subalgebra $\mathfrak{p}^\vee$ obtained from the nilradical $\mathfrak{n}^\vee$, and
\item $\mathbb{F}^\vee$ is determined by the parabolic subgroup $P^\vee < G^\vee$, with $G^\vee$ 
of the same type as $G$, with $\dim \mathbb{F}^\vee = \dim \mathbb{F}$.
\end{itemize}
\end{definition}

We now illustrate this definition  with examples.
We will use Malcev's notation, which works as follows. Denote the 2-form $e_i\wedge e_j$ by $e^{ij}$. So that,
for example,  a $6$-tuple $(0,0,0,e^{12}, e^{13},e^{14})$ describes the Lie algebra with dual 
generated by $e_1,\cdots, e_6$ such 
that $de_1=de_2=de_3=0$, $de_4= e_1\wedge e_2$, $de_5= e_1\wedge e_3$, and $de_5= e_1\wedge e_4$.

\begin{example}(A fixed point of the correspondence)\label{ex-sl3} Consider the maximal flag manifold  of the group $G^\mathbb{C} = SL(3,\mathbb{C})$, that is, the homogeneous manifold $\mathbb F=SL(3,\mathbb{C})/P$, where $P$ is a minimal parabolic subgroup. Note that $\mathfrak{g}^\mathbb{C} = \mathfrak{sl}(3,\mathbb{C})$, so we have the following data:
\begin{itemize}
\item[$\iota.$] the system of simple roots: $\Sigma = \{ \alpha, \beta\}$,
\item[$\iota\iota.$] the set of positive roots: $\Pi^+ = \{ \alpha, \beta, \alpha+\beta\}$,  
\item[$\iota\iota\iota.$] the root spaces: $\mathfrak{g}_\alpha = \langle E_{12}\rangle$, $\mathfrak{g}_\beta = \langle E_{23}\rangle$, $\mathfrak{g}_{\alpha+\beta} = \langle E_{13}\rangle$, and 
\item[$\iota\nu.$] the Cartan subalgebra: $\mathfrak{h} = \langle E_{11} - E_{22}, E_{22}-E_{33}\rangle$.
\end{itemize} 

Observe that we have the following nilpotent Lie algebra $\mathfrak{n} = \sum_{\gamma \in \Pi^+} \mathfrak{g}_{\gamma}$. Since $[E_{12},E_{23}] = E_{13}$ and all other brackets are zero, we have that $\mathfrak{a} = \langle E_{13} \rangle$ is a central ideal of $\mathfrak{n}$. 

Using  Malcev's notation, we write  $\mathfrak{n} = (0,0,-e^{12})$ and $\mathfrak{a} = ( e_3)$. 
Then $(\mathfrak{n},\mathfrak{a}, H= e^{123})$ is an admissible triple. Following Theorem \ref{nildual}  the dual triple 
$(\mathfrak{n}^\vee,\mathfrak{a}^\vee, H^\vee)$ is given by $\mathfrak{n}^\vee = (0,0,e^{12})$, 
$\mathfrak{a} = \mathbb{R}$ and $H^\vee = (-e^{123})$. Note that, by  Definition \ref{dualflag}, $\mathfrak{n}$ is self-dual and therefore $(\mathbb{F},H)$ corresponds to itself:


\begin{center}
\begin{figure}[H]
\begin{tikzpicture}
\node at (1,0) {$H=e^{123}$};
\node at (8.5,0) {$H^\vee=-e^{123}$};
\node at (1,3) {$\mn=(0,0,-e^{12}),\, \ma =(e_3) $};
\node at (8.5,3) {$\mn^\vee=(0,0,e^{12}),\, \ma =(\tilde e_3) $};
\draw [->] (1.1,2.85) to [out=-45, in=140](8.9,0.35);
\draw [dashed, ->] (1.4,0.35) to [out=40,in=-135](8.5,2.85);

\end{tikzpicture}
\caption{A fixed point of the correspondence.}
\end{figure}
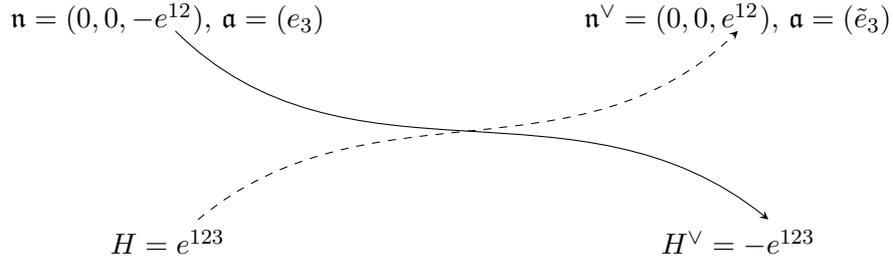
\end{center}
\end{example}

\begin{remark}
We observe that Example \ref{ex-sl3} considers again the same flag manifold of Example \ref{e1}, but with a different flux. We then observe that, while as a flowing manifold with flux $H=e^{123}$ it remains unaffected by the correspondence, when considered without flux it changes non-trivially.
\end{remark}

\begin{theorem}
Let $\mathbb{F}$ be a maximal flag manifold. There exists an invariant $3$-form $H$ such that
the flowing flag  $(\mathbb{F},H)$ is self-corresponding.
\end{theorem}
\begin{proof}
Let $\mathbb{F} = G/P$ be a maximal flag manifold and denote by $\mathfrak{n}$ the nilradical of $\mathfrak{p}=\Lie(P)$. Denote by $\delta$ the highest root of the Lie algebra $\mathfrak{g}$ of $G$ and by $X_\delta$ the generator of the root space $\mathfrak{g}_\delta$. Fix a basis $\{ e_1, e_2, \cdots , e_l\}$ of $\mathfrak{n}$ such that $e_l = X_\delta$. Thus, using Malcev's notation we can write $X_\delta = \sum_{i<j}h_{ij}e^{ij}$. By choosing $\mathfrak{a}=\langle e_l \rangle$ and $H = \left( \sum_{i<j}h_{ij}e^{ij} \right)\wedge e^l$ we have that $(\mathfrak{n}, \mathfrak{a}, H)$ is an admissible triple. The corresponding dual triple, satisfies $\tilde{e}_l = \sum_{i<j}h_{ij}e^{ij}$  and $H^\vee = \sum_{i<j}h_{ij}e^{ijl}$. So, $\mathfrak{n}^\vee \simeq \mathfrak{n}$, $\mathfrak{a}^\vee = \mathbb{R}$ and $H^\vee \simeq H$, therefore $(\mathfrak{n}, \mathfrak{a}, H)$ is selfdual.  
\end{proof}

We observe that a good choice of the abelian ideal $\mathfrak{a}$ is important for the correspondence. If $\mathfrak{a}$ turns out to be too small to provide the corresponding flowing flag, we try again with another choice of $\mathfrak{a}$. 

\begin{example}(Choosing  $\mathfrak{a}$ appropriately) This example illustrates the issue that a
correct choice of $\mathfrak a$ is essential for the correspondence.

Consider the maximal flag manifold $\mathbb{F}$ of $G^\mathbb{C} = SL(4,\mathbb{C})$. The parabolic subalgebra $\mathfrak{p}$ associated to $\mathbb{F}$ is the respective Borel subalgebra, and the nilradical of $\mathfrak{p}$ is given by $\mathfrak{n} = \sum_{\alpha\in \Pi^+} \mathfrak{g}_\alpha$. Using Malcev's notation we have
\[
\mathfrak{n} = (0,0,0, -e^{12}, -e^{23}, -e^{14}+e^{35}).
\]
Let $\mathfrak{a} = \langle e_6\rangle$ and $H = 0$, then we have
\[
\mathfrak{n}^\vee = (0,0,0,-e^{12},-e^{23},0), \ \mathfrak{a}^\vee = \mathbb{R} \ \textnormal{ and } \ H^\vee = -e^{146}+e^{356}.
\]
However, observe that there is no parabolic subalgebra $\mathfrak{p}^\vee$ such that its nilradical coincides with $\mathfrak{n}^\vee$. 
Indeed,  if  such a parabolic subalgebra $\mathfrak{p}^\vee$ existed, then the associated flag manifold $\mathbb{F}^\vee$ would 
have only two isotropy components, which is impossible. Therefore, we need a larger $\mathfrak{a}$ for the correspondence applied to $(\mathbb{F},0)$.

Now we try again choosing $\mathfrak{a} = \langle e_4,e_5,e_6\rangle$, and we obtain 
\[
\mathfrak{n}^\vee = (0,0,0,0,0,0), \ \mathfrak{a}^\vee = \mathbb{R}^3 \ \textnormal{ and } \ H^\vee = -e^{124}-e^{235}-e^{146}+e^{356}.
\]

{\small
\begin{center}
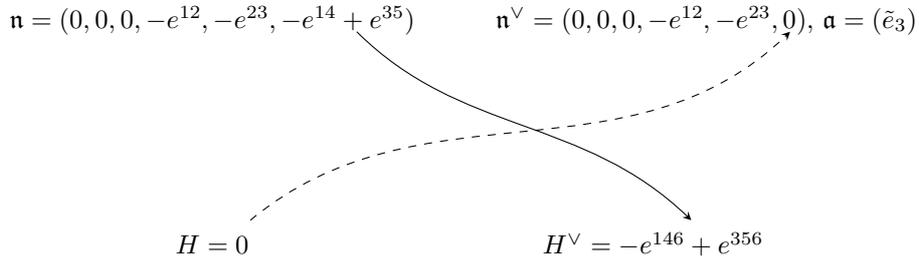
\begin{figure}[H]
\begin{tikzpicture}
\node at (1,0) {$H=0$};
\node at (6.8,0) {$H^\vee= -e^{146}+e^{356}$};
\node at (1,3) {$\mn=(0,0,0, -e^{12}, -e^{23}, -e^{14}+e^{35}) $};
\node at (7.5,3) {$\mn^\vee=(0,0,0,-e^{12},-e^{23},0),\, \ma =(\tilde e_3) $};
\draw [->] (2.9,2.85) to [out=-45, in=135](7.3,0.35);
\draw [dashed, ->] (1.5,0.35) to [out=40,in=-135](8.6,2.85);

\end{tikzpicture}
\caption{Appropriate choice of $\mathfrak{a}$.}
\end{figure}
\end{center}}
Since $\mathfrak{n}^\vee$ is abelian, it follows that 
\[
\mathbb{F}^\vee = SL(7,\mathbb{C})/P_\Theta \simeq SU(7)/S(U(6)\times U(1))
\] 
where $\Theta = \Sigma \backslash \{ \alpha_6\}$. So that $(\mathbb{F}^\vee,H^\vee)$ is the desired corresponding flag.
\end{example}

\begin{example}(Partial flag)
Consider the Lie algebra $\mathfrak{g}^\mathbb{C} = \mathfrak{sl}(6,\mathbb{C})$ and denote by $\Sigma = \{ \alpha_1,\alpha_2, \alpha_3, \alpha_4, \alpha_5\}$ its simple root system. Choosing the subset 
$\Theta = \{ \alpha_1, \alpha_3, \alpha_5\}$  of $\Sigma$,  we may construct the flag manifold 
$$\mathbb{F} = G^\mathbb{C}/P_{\Theta} \simeq SU(6)/S(U(2)\times U(2) \times U(2)).$$ 
The nilpotent algebra $\mathfrak{n}$ associated to $\mathbb{F}$ can be decomposed into three components, namely $\mathfrak{m}_i = \sum_{\alpha\in R_i}\mathfrak{g}_\alpha$ with $i \in \{1,2,3\}$, where 
$$R_1 = \{ \alpha \in \Pi^+ _M \ | \ \htt( \alpha_2) = 1 \textnormal{ and } \htt( \alpha_4 ) = 0\},$$ $$R_2 = \{ \alpha \in \Pi^+ _M \ | \ \htt( \alpha_2) = 0 \textnormal{ and } \htt( \alpha_4 ) = 1\},$$ and $$R_3 = \{ \alpha \in \Pi^+ _M \ | \ \htt( \alpha_2) = 1 \textnormal{ and } \htt( \alpha_4 ) = 1\}.$$ Setting 
$$\mathfrak m_1= \langle e_1,e_2,e_3,e_4\rangle,\mathfrak m_2= \langle e_5,e_6,e_7,e_8\rangle,
\mathfrak m_3= \langle e_9,e_{10},e_{11},e_{12}\rangle,$$ and using Malcev's notation 
we can write $$\mathfrak{n} = (0,0,0,0,0,0,0,0, e^{15}-e^{36}, e^{17}-e^{38}, e^{46}-e^{25}, e^{48}-e^{27}).$$ 
Since $[\mathfrak{m}_1, \mathfrak{m}_2] \subseteq \mathfrak{m}_3$ then $\mathfrak{a} = \mathfrak{m}_3$ is a central ideal of $\mathfrak{n}$, it follows that
 $(\mathfrak{n},\mathfrak{a}, H=0)$ is an admissible triple. By Theorem \ref{nildual} the dual triple $(\mathfrak{n}^\vee, \mathfrak{a}^\vee, H^\vee)$ is given by $\mathfrak{n}^\vee = (0,0,0,0,0,0,0,0,0,0,0,0)$, $\mathfrak{a}^\vee = \mathbb{R}^4$ and 
$$H^\vee = e^{159} - e^{369} + e^{17(10)} - e^{38(10)} + e^{46(11)} - e^{25(11)} + e^{48(12)} - e^{27(12)}.$$ Note that $\mathfrak{n}^\vee$ is an abelian algebra and it defines the flag manifold 
$$\mathbb{F}^\vee = SL(13,\mathbb{C})/P_\Theta \simeq SU(13)/S(U(12)\times U(1)) \simeq \mathbb{CP}^{12}$$ where $\Theta = \Sigma \backslash \{\alpha_{12} \}$. 
 Here we obtain the infinitesimal duality:
 
{\footnotesize
\begin{center}
\begin{figure}[H]
\begin{tikzpicture}
\node at (-3.1,0) {$H=0$};
\node at (4,0) {$H^\vee= \!e^{159} - e^{369} + e^{17(10)} - e^{38(10)} + e^{46(11)} - e^{25(11)} + e^{48(12)} - e^{27(12)}$};
\node at (0.8,3) {$\mn=(0,\ldots,0, e^{15}-e^{36}, e^{17}-e^{38}, e^{46}-e^{25}, e^{48}-e^{27}),\mathfrak{a}=\mathfrak m_3$};
\node at (7.8,3) {$\mn^\vee=(0),\, \ma^\vee =\mathbb R^4 $};
\draw [->] (-1,2.85) to [out=-45, in=140](0.5,0.35);
\draw [->] (0.55,2.85) to [out=-35,in=150] (2.7,0.35);
\draw [->] (1.9,2.85) to [out=-25,in=145] (5.3,0.35);
\draw [->] (3.25,2.85) to [out=-25,in=135] (7.9,0.35);
\draw [dashed, ->] (-2.65,0.35) to [out=35,in=-135](7.5,2.85);

\end{tikzpicture}
\caption{Correspondence  $\mathbb{F}=\frac{SU(6)}{S(U(2)\times U(2) \times U(2))}$ to $\mathbb{CP}^{12}$.}
\end{figure}
\end{center}}
 
\end{example}

The previous examples generalize  to show that a flag manifold with three isotropy summands
equipped with $H=0$ corresponds to $\mathbb{CP}^r$ where $r$ is the dimension of $\mathbb{F}$. \\

\noindent {\bf Flag manifolds with three isotropy summands:}\label{sec:flag3somandos} These are generalized flag manifolds whose  isotropy representation decomposes into three pairwise
non-isomorphic irreducible components, i.e., $\mathfrak{m}=\mathfrak{m}_{1}\oplus\mathfrak{m}_{2}\oplus\mathfrak{m}_{3}$.
They are classified by Kimura in \cite{K} and described in terms of a choice of $\Theta\subset\Sigma$ as in Table \ref{table:3compo-sigma}.

\begin{table}[!htb] 

\centering

\renewcommand*{\arraystretch}{1.2}
\begin{tabular}{c|c} 


Type & $\Theta \subset \Sigma$\\ 
\hline

I & $\Theta=\Sigma\setminus\{\alpha_{p}:\htt(\alpha_{p})=3\}$\\
II & $\Theta=\Sigma\setminus \{\alpha_{p},\alpha_{q} : \htt(\alpha_{p})=\htt(\alpha_{q})=1\}$\\
\end{tabular}
\label{tab2}
\caption{Types of flag manifolds with three summands}\label{table:3compo-sigma} 

\end{table}

There exist exactly 7 flag manifolds with three summands of type I; they are quotients of the 
exceptional groups $G_2$, $F_4$, $E_6,E_7,E_8$, see \cite{K}. Here we will concentrate on those 
that are of type II instead; these are infinitely many. Denoting by  $d_{i}=\dim(\mathfrak{m}_{i})$
the dimensions of the individual components, we have:

\begin{table}[ht] 

\centering
\renewcommand*{\arraystretch}{2.15}
\begin{tabular}{c|c|c|c} 

Flag manifold & $d_{1}$ & $d_{2}$ & $d_{3}$\\

\hline
$SU(l+m+n)/S(U(l)\times U(m)\times U(n)) $ & $2mn$ & $2mp$ & $2np$\\
$SO(2l)/U(1)\times U(l-1), \quad  (l \geq 4)$ & $2(l-1)$ & $2(l-1)$ & $(l-1)(l-2)$\\
$E_{6}/SO(8)\times U(1)\times U(1)$& 16 & 16 & 16\\

\end{tabular}
\caption{Flag manifolds with three summands of type II}\label{table:3comp-tipo2} 
\label{tab2}
\end{table}\FloatBarrier

\begin{theorem}
Let $\mathbb{F} = SL(d,\mathbb{C})/P_\Theta$ be a flag manifold with three isotropy summands of dimensions $d_1, d_2$ and $d_3$ without flux. The corresponding flowing flag is $\mathbb{CP}^{d_1 + d_2 +d_3}$ equipped with a nontrivial flux $H^\vee$.
\end{theorem}

\begin{proof}
Since flag manifold $\mathbb{F} = SL(d,\mathbb{C})/P_\Theta$ has three isotropy summands, we have $\Theta=\Sigma\setminus \{\alpha_{p},\alpha_{q} : \htt(\alpha_{p})=\htt(\alpha_{q})=1\}$. If we denote $$R_1 = \{ \alpha \in \Pi^+ _M \ | \ \htt( \alpha_p) = 1 \textnormal{ and } \htt( \alpha_q ) = 0\}$$ $$R_2 = \{ \alpha \in \Pi^+ _M \ | \ \htt( \alpha_p) = 0 \textnormal{ and } \htt( \alpha_q ) = 1\}$$ $$R_3 = \{ \alpha \in \Pi^+ _M \ | \ \htt( \alpha_p) = 1 \textnormal{ and } \htt( \alpha_q ) = 1\},$$ then the isotropy summands are given by $$\mathfrak{m}_i = \sum_{\alpha \in R_i}\mathfrak{g}_\alpha$$ where $i=1,2,3$. So, it is simple to check that $[\mathfrak{m}_1,\mathfrak{m}_2]\subseteq \mathfrak{m}_3$ and the other brackets are zero. Then, using Malcev's notation we have
\[
\mathfrak{n} = (\underbrace{0,\cdots,0}_{d_1},\underbrace{0,\cdots, 0}_{d_2},\underbrace{*, \cdots, *}_{d_3}),
\]
where $d_i$ is the dimension of $\mathfrak{m}_i$. Therefore, putting $\mathfrak{a} = \mathfrak{m}_3$ and $H=0$ we have that
\[
\mathfrak{n}^\vee = (0,0,\cdots, 0 ,0), \ \mathfrak{a}^\vee = \mathbb{R}^{d_3} \ \textnormal{ and } \ H^\vee \neq 0,
\]
since $H^\vee$ is constructed from the $d_3$ nonzero coordinates of $\mathfrak{n}$; there are
$d_3$ nonzero such coordinates, so that 
according to Definition\thinspace\ref{dualflag}, $(\mathbb F^\vee, H^\vee)$ is the corresponding flowing flag of $(\mathbb F,0)$. Thus, $\mathfrak{n}^\vee$ is an abelian subalgebra of complex dimension $d_1 + d_2 + d_3$, and we conclude that
\[
\mathbb{F}^\vee = SL(d_1+d_2+d_3+1)/P ^\vee \simeq SU(d_1+d_2+d_3+1)/S(U(d_1+d_2+d_3)\times U(1)).
\]
\end{proof}

\begin{remark}
Observe that an  analogous statement does not hold for the other type II
 flag manifolds with three isotropy summands appearing
in Table\thinspace\ref{tab2}.  Indeed, for  
$\Theta=\Sigma\setminus \{\alpha_{p},\alpha_{q} : \htt(\alpha_{p})=\htt(\alpha_{q})=1\}$, the complex dimension of the flag 
 $\mathbb{F} = SO(2n,\mathbb{C})/P_\Theta$ 
 is $\frac{1}{2}(n^2+n-2)$, whereas
  the complex dimension of the Hermitian symmetric space $SO(2l)/U(1)\times U(l-1)$ is $\frac{1}{2}(l^2-l)$. 
  But, for our conjectured correspondence we would need  $$n^2+n-2 = l^2-l$$ for $l,n$ natural numbers ($n\geq 4$)
  which has no natural (or integer) solution. 

The case of the flag manifold of $E_6$ is similar. We have  complex dimension 24 for $E_6/ SO(8)\times U(1)\times U(1)$,
whereas the Hermitian symmetric space $E_6/SO(10)\times SO(2)$ has complex dimension $16$. 

\end{remark}

We now illustrate the fact that a single flowing flag may be taken to different corresponding flowing flags, 
according to the choice of dual parabolic $\mathfrak p^\vee$.

\begin{example}(Non-uniqueness)
Consider the maximal flag manifold $\mathbb{F}$ of $G^\mathbb{C} = SL(4,\mathbb{C})$. Since we are considering the maximal flag manifold, we have that the parabolic subalgebra $\mathfrak{p}$ associated to $\mathbb{F}$ is the respective Borel subalgebra. Then, the nilradical of $\mathfrak{p}$ is given by $\mathfrak{n} = \sum_{\alpha\in \Pi^+} \mathfrak{g}_\alpha$. Using Malcev's notation we have
\[
\mathfrak{n} = (0,0,0, -e^{12}, -e^{23}, -e^{14}+e^{35}).
\]
Given $(\mathbb{F},H=0)$ and considering $\mathfrak{a} = \langle e_4, e_5, e_6\rangle$, we have 
\[
\mathfrak{n}^\vee = (0,0,0,0,0,0), \ \mathfrak{a}^\vee = \mathbb{R}^3 \ \textnormal{ and } \ H^\vee = -e^{124}-e^{235}-e^{146}+e^{356}.
\]
{\footnotesize
\begin{center}
\begin{figure}[H]
\begin{tikzpicture}
\node at (0.7,0) {$H=0$};
\node at (8.2,0) {$H^\vee=  -e^{124}-e^{235}-e^{146}+e^{356}$};
\node at (1.8,3) {$\mn=(0,0,0, -e^{12}, -e^{23}, -e^{14}+e^{35}) $};
\node at (8.5,3) {$\mn^\vee=(0,0,0,0,0,0)$};
\draw [->] (1.25,2.85) to [out=-50, in=135](7.3,0.35);
\draw [->] (2.15,2.85) to [out=-50, in=133](8.2,0.35);
\draw [->] (3.5,2.85) to [out=-50, in=135](9.6,0.35);
\draw [dashed, ->] (1.1,0.35) to [out=40,in=-135](9.3,2.85);
\draw [dashed, ->] (1.1,0.35) to [out=40,in=-135](9,2.85);
\draw [dashed, ->] (1.1,0.35) to [out=40,in=-135](9.6,2.85);

\end{tikzpicture}
\caption{A single duality underlying 2 correspondences.}
\end{figure}
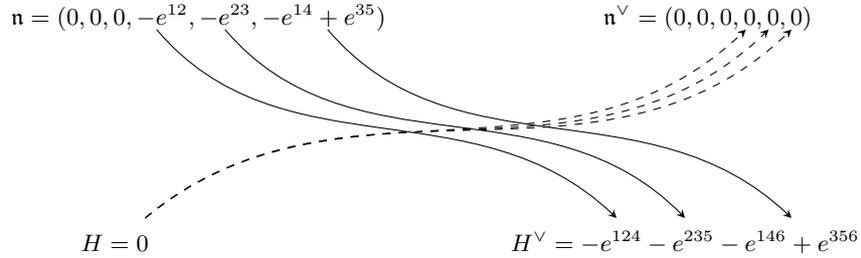
\end{center}}
Since $\mathfrak{n}^\vee$ is abelian, as a corresponding flowing flag we may choose
\[
\mathbb{F}_1 ^\vee = SL(7,\mathbb{C})/P_{\Theta_1} \simeq SU(7)/S(U(6)\times U(1))
\] 
where $\Theta_1 = \Sigma \backslash \{ \alpha_6\}$. Alternatively, we may  choose
\[
\mathbb{F}_2 ^\vee = SL(5,\mathbb{C})/P_{\Theta_2} \simeq SU(5)/S(U(3)\times U(2))
\] 
where $\Theta_2 = \Sigma \backslash \{ \alpha_3\}$. We observe that the flags 
$\mathbb{F}_1 ^\vee$ and $\mathbb{F}_2 ^\vee$ are not isomorphic, hence the 
in this example we see that diagram \ref{fig-def} may be completed by more than one choice 
of flowing flag.
\end{example}

\section{Exchanging complex and symplectic structures}

To explore how our correspondence affects generalized complex structures we will use the characterization of invariant structures obtained in \cite{VS}; which tells us that we can decompose an invariant generalized complex structure $\mathcal{J}$ on $\mathbb{F}$ in terms of the root spaces of $\mathfrak{g}$, that is, $\mathcal{J} = \bigoplus_{\alpha>0} \mathcal{J}_\alpha$. Moreover, for each root we have only two possibilities:
\begin{itemize}
\item[a)] Complex type
\[
\pm \left( \begin{array}{cccc}
0 & -1 & 0 & 0\\
1 & 0 & 0 & 0\\
0 & 0 & 0 & -1\\
0 & 0 & 1 & 0
\end{array}\right)
\]

\item[b)] Noncomplex type ($B$-transformation of symplectic type, see \cite{GVV})
\[
\left( \begin{array}{cccc}
a_\alpha & 0 & 0 & -x_\alpha\\
0 & a_\alpha & x_\alpha & 0\\
0 & -y_\alpha & -a_\alpha & 0\\
y_\alpha & 0 & 0 & -a_\alpha
\end{array}\right)
\]
where $a_\alpha ^2 -x_\alpha y_\alpha = -1$.
\end{itemize}

Given corresponding pairs of flowing flags $(\mathbb{F},H)$ and $(\mathbb{F}^\vee,H^\vee)$, we use the isomorphism $\varphi \colon T_{e} \mathbb{F} \oplus T^\ast _{e}\mathbb{F} \rightarrow T_{e'}\mathbb{F}^\vee \oplus T^\ast _{e'}\mathbb{F}^\vee$ preserving the Courant bracket and the canonical bilinear form defined by Cavalcanti and Gualtieri \cite{CG1,CG2}. 
Therefore, combining the Cavalcanti--Gualtieri map and infinitesimal $T$-duality, we transport generalized complex structures between corresponding flowing flags.

\begin{proposition}
Let $(\mathbb{F},H)$ and $(\mathbb{F}^\vee,H^\vee)$ be corresponding flowing flags. If $\mathcal{J}$ is an invariant generalized almost complex structure on $(\mathbb{F},H)$ then
\[
\widetilde{\mathcal{J}} = \varphi \circ \mathcal{J} \circ \varphi^{-1}
\]
is an invariant generalized almost complex structure on $(\mathbb{F}^\vee,H^\vee)$.
\end{proposition}

\begin{proof} 

Given that $(\mathbb{F},H)$ and $(\mathbb{F}^\vee,H^\vee)$ are corresponding flowing flags, by Definition \ref{dualflag}, there exist $\mathfrak{a}$ and $\mathfrak{a}^\vee$, such that $(\mathfrak{n},\mathfrak{a}, H)$ and $(\mathfrak{n}^\vee, \mathfrak{a}^\vee, H^\vee)$ are admissible triples, where $\mathfrak{n}\approx T_{b_0}\mathbb{F}$ and $\mathfrak{n}^\vee \approx T_{b_0}\mathbb{F}^\vee$.

Let $\mathcal{J}$ be an invariant generalized almost complex structure on $(\mathbb{F},H)$. As a consequence of invariance we have that $\mathcal{J}$ is completely determined by its restriction to the origin, named $J$, and $\widetilde{J} = \varphi \circ J \circ \varphi^{-1}$ is a generalized almost complex structure on $(\mathfrak{n}^\vee,H^\vee)$ which extends to define an invariant generalized almost complex structure  $\widetilde{\mathcal{J}}$ on $(\mathbb{F}^\vee,H^\vee)$.
\end{proof}

\begin{remark}
The isomorphism $\varphi$ is defined on the Lie algebra and, following \cite[Cor.\thinspace 3.9]{BGS}, if we consider basis $$\{ y_1,\cdots ,y_t, x_1,\cdots ,x_m\} \quad \textnormal{ and } \quad \{ y_1,\cdots ,y_t, z_1,\cdots ,z_m\}$$ of $\mathfrak{m}$ and $\mathfrak{m}^\vee$ where $x_1,\cdots x_t$ is a basis of $\mathfrak{a}$ and $z_1,\cdots ,z_m$ is a basis of $\mathfrak{a}^\vee$ the matrix of the isomorphism $\varphi$ has the form
\[
\varphi = \left( \begin{array}{cccc}
1_{t\times t} & 0 & 0 & 0\\
0 & 0 & 0 & -1_{m\times m}\\
0 & 0 & 1_{t\times t} & 0\\
0 & -1_{m\times m} & 0 & 0
\end{array}\right).
\]
\end{remark}
%

\begin{corollary}
Within each root space, our correspondence interchanges  types of generalized complex structures, that is, almost complex structures are exchanged with those of noncomplex type ($B$-transformations of symplectic structures).
\end{corollary}
\begin{proof}
Given two corresponding flowing flags, since an invariant generalized complex structure decomposes as $\mathcal{J} = \bigoplus_{\alpha>0} \mathcal{J}_\alpha$, let us compute compute $\varphi \circ \mathcal{J}_\alpha \circ \varphi^{-1}$ where $\mathcal{J}_\alpha$ is of complex and noncomplex type. First, let us consider $\mathcal{J}_\alpha$ of complex type.
\begin{eqnarray*}
\widetilde{\mathcal{J}_\alpha} & = & \varphi \circ \mathcal{J}_\alpha \circ \varphi^{-1}\\
& = & \left( \begin{array}{cccc}
1 & 0 & 0 & 0\\
0 & 0 & 0 & -1\\
0 & 0 & 1 & 0\\
0 & -1 & 0 & 0
\end{array}\right)
\left( \begin{array}{cccc}
0 & -1 & 0 & 0\\
1 & 0 & 0 & 0\\
0 & 0 & 0 & -1\\
0 & 0 & 1 & 0
\end{array}\right)
\left( \begin{array}{cccc}
1 & 0 & 0 & 0\\
0 & 0 & 0 & -1\\
0 & 0 & 1 & 0\\
0 & -1 & 0 & 0
\end{array}\right) \\
& = & \left( \begin{array}{cccc}
0 & 0 & 0 & 1\\
0 & 0 & -1 & 0\\
0 & 1 & 0 & 0\\
-1 & 0 & 0 & 0
\end{array}\right)
\end{eqnarray*}
which is a generalized complex structure of symplectic type.

Now consider $\mathcal{J}_\alpha$ of noncomplex type. 
\begin{eqnarray*}
\widetilde{\mathcal{J}_\alpha} & = & \varphi \circ \mathcal{J} \circ \varphi^{-1}\\
& = & \left( \begin{array}{cccc}
1 & 0 & 0 & 0\\
0 & 0 & 0 & -1\\
0 & 0 & 1 & 0\\
0 & -1 & 0 & 0
\end{array}\right)
\left( \begin{array}{cccc}
a_\alpha & 0 & 0 & -x_\alpha\\
0 & a_\alpha & x_\alpha & 0\\
0 & -y_\alpha & -a_\alpha & 0\\
y_\alpha & 0 & 0 & -a_\alpha
\end{array}\right)
\left( \begin{array}{cccc}
1 & 0 & 0 & 0\\
0 & 0 & 0 & -1\\
0 & 0 & 1 & 0\\
0 & -1 & 0 & 0
\end{array}\right) \\
& = & \left( \begin{array}{cccc}
a_\alpha & x_\alpha & 0 & 0\\
-y_\alpha & -a_\alpha & 0 & 0\\
0 & 0 & -a_\alpha & y_\alpha\\
0 & 0 & -x_\alpha & a_\alpha
\end{array}\right)\\
\end{eqnarray*}
which is a generalized complex structure of complex type, since we have that $a_\alpha ^2 -x_\alpha y_\alpha = -1$. 
\end{proof}


\begin{remark}
Observe that the integrability of generalized complex structures is not necessarily preserved by $\varphi$. For a fixed irreducible isotropy component of a flag manifold, \cite{V} proved that the invariant generalized complex structures must have the same `type'. Hence, for every root space in this isotropy component we must have all $\mathcal{J}_\alpha$ of either complex or noncomplex type simultaneously. 

Now, consider the maximal flag manifold $\mathbb{F}$ of $SL(3,\mathbb{C})$. We know from Example \ref{e1} that $(\mathbb{F},0)$ and $(\mathbb{F}^\vee,H^\vee)$ are corresponding flowing flags, where $\mathbb{F}^\vee = SL(4,\mathbb{C})/P_\Theta$ with $\Theta = \{ \alpha_1,\alpha_2\}$ and $H^\vee = -e^{123}$. 

If we consider $\mathcal{J}$ an invariant generalized complex structure on $\mathbb{F}$ such that $\mathcal{J}_\alpha$ is of symplectic type and $\mathcal{J}_\beta$ and $\mathcal{J}_{\alpha +\beta}$ are of complex type, we know from \cite{VS} that $\mathcal{J}$ is integrable. But, the corresponding structure on $\mathbb{F}^\vee$, given by $\widetilde{\mathcal{J}} = \varphi \circ \mathcal{J} \circ \varphi^{-1}$, is a generalized complex structure such that $\mathcal{J}_\alpha$ is of complex type and $\mathcal{J}_\beta$ and $\mathcal{J}_{\alpha +\beta}$ are of symplectic type. Since $\mathbb{F}^\vee$ is a partial flag manifold with only one irreducible isotropy component, we would have complex and symplectic type inside the same isotropy component, thus $\widetilde{\mathcal{J}}$ is not integrable.
\end{remark}

\end{document}